\documentclass{amsart}
\usepackage{amssymb, amsmath}
\usepackage{amsthm}
\usepackage{amscd}
\usepackage{graphicx}
\newtheorem{theorem}{Theorem}

\theoremstyle{definition}
\newtheorem*{remark}{Remark}
\newtheorem*{acknowledgement}{Acknowledgement}

\title{Nonpositively curved manifolds containing a prescribed nonpositively curved hypersurface}
\author{T. T$\hat{\mathrm{a}}$m Nguy$\tilde{\hat{\mathrm{e}}}$n Phan}
\address{Department of Mathematics\\
5734 S. University Ave.\\
Chicago, IL 60637}
\email{ttamnp@math.uchicago.edu}

\def\R{\mathbb{R}}

\def\M{\widehat{M}}

\input xy
\xyoption{all}
\begin{document}
\begin{abstract}
We use pinched smooth hyperbolization to show that every closed, nonpositively curved $n$-dimensional manifold $M$ can be embedded as a totally geodesic submanifold of a closed, nonpositively curved $(n+1)$-dimensional manifold $\widehat{M}$ of geometric rank one.
\end{abstract}
\maketitle
Ralf Spatzier asked the author the following interesting question: for a closed manifold $M$ with sectional curvature $\leq 0$ (e.g. a closed, nonpositively curved, locally symmetric manifold), is there a closed manifold $\M$ of one dimension higher with sectional curvature $\leq 0$ and has geometric rank $1$ (and thus is not a product) that contains $M$ as a totally geodesic submanifold? The answer to this question is yes thanks to recent technology of pinched smooth hyperbolization (\cite{Ontanedahyperbolization}). In this paper we give a construction of such a manifold $\M$.
\begin{theorem}
Let $(M, g_M)$ be a closed, Riemannian manifold of dimension $n$ with sectional curvature $\kappa(M) \leq 0$. There exist a closed, Riemannian $(n+1)$-dimensional manifold $\M$ of geometric rank $1$ with sectional curvature $\kappa(\M) \leq 0$ and a isometric embedding $f\colon M \longrightarrow \M$. 
\end{theorem}
\begin{proof}
Let $\bigtriangleup$ be a smooth triangulation of $M$. We extend $\bigtriangleup$ to a triangulation of $M\times [0,1]$. We cone off the boundary of $M\times [0,1]$ (which has two components) and denote the resulting simplicial complex by $X$. Then $X$ is a manifold with one singular cone point $*$, that is, $X\setminus\{*\}$ is a manifold. Let $h(X)$ be a strict hyperbolization of $X$ (\cite{stricthyperbolization}). Then $h(X)$ is a manifold with one singularity $h(*)$. We pick $h(X)$ such that the faces of each Charney-Davis hyperbolization piece have large enough width as in \cite[Lemma 9.1.1]{Ontanedahyperbolization} so that pinched smooth hyperbolization can be applied to $h(X)\setminus\{h(*)\}$.

Let $W = h(X)\setminus\{h(*)\}$. Then $W$ is a noncompact manifold with two ends, each of which is homeomorphic to $M\times(0,\infty)$. Using the same proof given in \cite[Section 11]{Ontanedahyperbolization} there is a Riemannian metric $g$ on $W$ with sectional curvature $<0$ with the property that each end (with metric $g$) is isometric to $M\times (a,\infty)$ with metric
\[dt^2 + e^{-2t}g_M.\]
The actual value of $a$ is not crucial in this argument, so we assume $a<-1$. (To be able to apply the method in \cite[Section 11]{Ontanedahyperbolization} it is required that the Whitehead group of M be trivial if $M$ has dimension $>4$ (\cite[Theorem 7.9.1]{Ontanedahyperbolization}). But, since M has a non-positively curved metric $g_M$; this is true by a result of Farrell and Jones \cite{FJ4}.)

We truncate each of end of $W$ at $t = 0$ and glue the two boundary components of the resulting manifold together. We then get a closed manifold $\M$ with a Riemannian metric $\overline{g}$ that is not smooth at the gluing. The metric $\overline{g}$ is a warped product $dt^2 + e^{-2|t|}g_M$, for $-1<t<1$. Therefore, in order to smooth out the metric $\overline{g}$, we just need to smooth out the warping function $e^{-2|t|}$ around $t=0$ without altering the nonpositivity of the curvature.

Observe that since the metric $g_M$ is nonpositively curved, the warped product metric $dt^2 + \phi^2(t)g_M$ on $\R\times M$ has nonpositive curvature if $\phi(t)$ is a convex function by the Bishop-O'Neill curvature formula (\cite{BishopONeill}). Thus we can pick $\phi$ to be a convex, smooth, even function that agrees with $e^{-2|t|}$ outside a small neighbourhood of $t=0$ and assumes a minimum at $t=0$. We then obtain a Riemannian metric $\widehat{g}$ on $\M$ that has sectional curvature $\kappa \leq 0$. 

It is not hard to see that map $f\colon M \longrightarrow \M$ defined by identifying $M$ with cross section $t=0$ is a isometric embedding due to the evenness of $\phi(t)$.
\end{proof}
\begin{remark}
The theorem holds if we replace ``$\leq$" by ``$<$".
\end{remark}
\begin{acknowledgement}
The author would like to thank the anonymous referee for useful comments and for pointing out that the Whitehead group of $M$ must be trivial for the argument to work.
\end{acknowledgement}

\bibliographystyle{amsplain}
\bibliography{bibliography}

\end{document}